\documentclass{elsarticle}

\usepackage{amssymb,amsmath,amsthm}
\usepackage{tikz}

\usepackage{fixltx2e}
\MakeRobust{\overrightarrow}

\usepackage[left=3.5cm,right=3.5cm,top=3cm,bottom=3cm]{geometry}

\newtheorem{theorem}{Theorem}[section]

\newtheorem{conjecture}[theorem]{Conjecture}
\newtheorem{corollary}[theorem]{Corollary}

\newtheorem{lemma}[theorem]{Lemma}

\newtheorem{remark}[theorem]{Remark}

\begin{document}

\title{$k$-colored kernels in semicomplete multipartite digraphs}

\author[unam]{Hortensia Galeana--S\'{a}nchez}
\ead{hgaleana@matem.unam.mx}

\author[uam]{Bernardo Llano\corref{cor1}}
\ead{llano@xanum.uam.mx}

\author[unam]{Juan Jos\'{e} Montellano--Ballesteros}
\ead{juancho@matem.unam.mx}

\address[unam]{Instituto de Matem\'{a}ticas, UNAM, Ciudad Universitaria, 04510, M\'{e}xico, D. F.}

\address[uam]{Departamento de Matem\'{a}ticas, Universidad Aut\'{o}noma Metropolitana, Iztapalapa, San Rafael Atlixco 186,
               Colonia Vicentina, 09340, M\'{e}xico, D.F.}

\cortext[cor1]{Corresponding author.}


\begin{abstract}

An $m$-colored digraph $D$ has $k$-colored kernel if there exists a
subset $K $ of its vertices such that for every vertex $v\notin K$
there exists an at
most $k$-colored directed path from $v$ to a vertex of $K$ and for every $%
u,v\in K$ there does not exist an at most $k$-colored directed path
between them. In this paper we prove that an $m$-colored
semicomplete $r$-partite digraph $D$ has a $k$-colored kernel
provided that $r\geq 3$ and

\begin{enumerate}
\item[(i)] $k\geq 4,$

\item[(ii)] $k=3$ and every $\overrightarrow{C}_{4}$ contained in $D$ is at
most $2$-colored and, either every $\overrightarrow{C}_{5}$
contained in $D$ is at most $3$-colored or every
$\overrightarrow{C}_{3}\uparrow \overrightarrow{C}_{3}$ contained in
$D$ is at most $2$-colored,

\item[(iii)] $k=2$ and every $\overrightarrow{C}_{3}$ and $\overrightarrow{C}%
_{4}$ contained in $D$ is monochromatic.
\end{enumerate}

If $D$ is an $m$-colored semicomplete bipartite digraph and $k=2$
(resp. $k=3 $) and every $\overrightarrow{C}_{4}\upuparrows
\overrightarrow{C}_{4}$
contained in $D$ is at most $2$-colored (resp. $3$-colored), then $D$ has a $%
2$-colored (resp. $3$-colored) kernel. Using these and previous
results, we obtain conditions for the existence of $k$-colored
kernels in $m$-colored semicomplete $r$-partite digraphs for every
$k\geq 2$ and $r\geq 2$.

\end{abstract}


\maketitle

\section{Introduction}

Let $m,$ $j$ and $k$ positive integers. A digraph $D$ is said to be $m$%
\textit{-colored} if the arcs of $D$ are colored with $m$ colors. Given $%
u,v\in V(D),$ a directed path from $u$ to $v$ of $D,$ denoted by
$u\leadsto
v,$ is $j$\textit{-colored }if all its arcs\textit{\ }use exactly\textit{\ }$%
j$ colors and it is represented by $u\leadsto _{j}v.$ When $j=1,$
the
directed path is said to be \textit{monochromatic. }A nonempty set $%
S\subseteq V(D)$ is a $k$\textit{-colored} \textit{absorbent set }if
for every vertex $u\in V(D)-S$ there exists $v\in S$ such that
$u\leadsto _{j}v$
with $1\leq j\leq k.$ A nonempty set $S\subseteq V(D)$ is a called a $k$%
\textit{-colored independent} \textit{set }if for every $u,v\in S$
there
does not exist $u\leadsto _{j}v$ with $1\leq j\leq k.$ Let $D$ be an $m$%
-colored digraph. A set $K\subseteq V(D)$ is called a
$k$\textit{-colored kernel }if $K$ is a $k$-colored absorbent and
independent set. This definition was introduced in \cite{MCh}, where
the first basic results were proved. We observe that a $1$-colored
kernel is a kernel by monochromatic directed paths, a notion that
has widely studied in the literature, see for
instance \cite{GS1}, \cite{GS2}, \cite{GS-Ll-MB2}, \cite{GS-RM1}, \cite%
{GS-RM2}, \cite{SSW}, \cite{Ming} and \cite{Wloch}.

An arc $(u,v)\in A(D)$ is \textit{asymmetric }(resp. \textit{symmetric}) if $%
(v,u)\notin A(D)$ (resp. $(v,u)\in A(D)$). We denote by $\overrightarrow{C}%
_{n}$ the directed cycle of length $n.$ A semicomplete $r$-partite
digraph $D $ with $r\geq 2$ is an orientation of an $r$-partite
complete graph in which
symmetric arcs are allowed. A digraph $D$ is called $3$-\textit{%
quasi-transitive }if whenever distinct vertices
$u_{0},u_{1},u_{2},u_{3}\in V(D)$ such that $u_{0}\longrightarrow
u_{1}\longrightarrow u_{2}\longrightarrow u_{3}$ there exists at
least $(u_{0},u_{3})\in A(D)$ or
$(u_{3},u_{0})\in A(D).$ In particular, bipartite semicomplete digraphs are $%
3$-quasi-transitive$.$

Let $D^{\prime }$ a subdigraph of an $m$-colored digraph $D.$ We say that $%
D^{\prime }$ is \textit{monochromatic }if every arc of $D^{\prime }$
is
colored with the same color and $D^{\prime }$ is \textit{at most }$k$\textit{%
-colored }if the arcs of $D^{\prime }$ are colored with at most $k$
colors. In this paper, we particularly use subdigraphs of
semicomplete $r$-partite
digraph which are at most $2$- and $3$-colored. We defined the digraphs $%
\overrightarrow{C}_{3}\uparrow \overrightarrow{C}_{3}$ (resp. $%
\overrightarrow{C}_{4}\upuparrows \overrightarrow{C}_{4}$) as two
directed cycles $\overrightarrow{C}_{3}$ (resp.
$\overrightarrow{C}_{4}$) joined by an arc (resp. by two consecutive
arcs), see the next picture.

\begin{figure}[h]
\begin{center}
\begin{tikzpicture}[line width=0.8pt]
    \shade[ball color=black] (0,1)  circle (.10cm);
    \shade[ball color=black] (-1,0) circle (.10cm);
    \shade[ball color=black] (0,-1) circle (.10cm);
    \shade[ball color=black] (1,0)  circle (.10cm);
    \draw[-stealth](-0.8,0.2)--(-0.2,0.8);
    \draw[-stealth](0.2,0.8)--(0.8,0.2);
    \draw[-stealth](-0.8,-0.2)--(-0.2,-0.8);
    \draw[-stealth](0.2,-0.8)--(0.8,-0.2);
    \draw[stealth-](-.75,0)--(0.75,0);
    \shade[ball color=black] (4,1) circle (.10cm);
    \shade[ball color=black] (3,0) circle (.10cm);
    \shade[ball color=black] (4,-1)  circle (.10cm);
    \shade[ball color=black] (5,0) circle (.10cm);
    \shade[ball color=black] (4,0) circle (.10cm);
    \draw[stealth-](3.2,0.2)--(3.8,0.8);
    \draw[stealth-](3.2,-0.2)--(3.8,-0.8);
    \draw[stealth-](4.2,0.8)--(4.8,0.2);
    \draw[stealth-](4.2,-0.8)--(4.8,-0.2);
    \draw[-stealth](3.25,0)--(3.75,0);
    \draw[-stealth](4.25,0)--(4.75,0);
\end{tikzpicture}
\end{center}

\caption{$\overrightarrow{C}_{3}\uparrow \overrightarrow{C}_{3}$ and
$\overrightarrow{C}_{4}\upuparrows \overrightarrow{C}_{4}$,
respectively.}

\label{figf}
\end{figure}

The goal of this work is to complete the study of the existence of
$k$-colored kernels in semicomplete $r$-partite digraphs for every
$k \geq 2$. The problem for $1$-colored kernels in bipartite
tourrnaments was studied in \cite{GS-RM1}. In that paper, the
authors proved that a if every $\overrightarrow{C}_{4}$ contained in
an $m$-colored bipartite tournament $T$ is monochromatic, then $T$
has a $1 $-colored kernel. Let $r\geq 3.$ In \cite{GS-RM2}, it was
proved that if every $\overrightarrow{C}_{3}$ and
$\overrightarrow{C}_{4}$ contained in a $r $-partite tournament $T$
is monochromatic then $T$ has a $1$-colored kernel. In
\cite{GS-Ll-MB} among other results, we showed that $m$-colored
quasi-transitive and $3$-quasi-transitive digraphs have a $k
$-colored kernel for every $k\geq 3$ and $k\geq 4,$ respectively. As
a consequence, $m$-colored semicomplete bipartite digraphs have a
$k$-colored kernel for every $k\geq 3$ and $k\geq 4,$ respectively.

In this paper we prove that an $m$-colored semicomplete $r$-partite
digraph $D$ has a $k$-colored kernel provided that $r\geq 3$ and

\begin{enumerate}
\item[(i)] $k\geq 4,$

\item[(ii)] $k=3$ and every $\overrightarrow{C}_{4}$ contained in $D$ is at
most $2$-colored and, either every $\overrightarrow{C}_{5}$
contained in $D$ is at most $3$-colored or every
$\overrightarrow{C}_{3}\uparrow \overrightarrow{C}_{3}$ contained in
$D$ is at most $2$-colored,

\item[(iii)] $k=2$ and every $\overrightarrow{C}_{3}$ and $\overrightarrow{C}%
_{4}$ contained in $D$ is monochromatic.
\end{enumerate}

If $D$ is an $m$-colored semicomplete bipartite digraph and $k=2$
(resp. $k=3 $) and every $\overrightarrow{C}_{4}\upuparrows
\overrightarrow{C}_{4}$
contained in $D$ is at most $2$-colored (resp. $3$-colored), then $D$ has a $%
2$-colored (resp. $3$-colored) kernel. Using these and previous
results, we obtain conditions for the existence of $k$-colored
kernels in $m$-colored semicomplete $r$-partite digraphs for every
$k\geq 2$ and $r\geq 2$ (see Corollary \ref{final-cor}). If we are
restricted to the family of the $m$-colored multipartite
tournaments, then we have conditions for the existence of
$k$-colored kernels for every $k\geq 1$ and $r\geq 2$ using the main
results of this paper and those obtained in \cite{GS-RM1} and
\cite{GS-RM2} (see Corollary \ref{final-cor2}).

We finish this introduction including some simple definitions and a
well-known result that will be useful in proving the main re sults.

Let $D$ be a digraph and $x,y\in V(D).$ The \textit{distance from}
$x$ \textit{to} $y,$ denoted by $d(x,y)$ is the minimum length
(number of arcs) of a $x\leadsto y.$

Recall that a \textit{kernel }$K$ of $D$ is an independent set of
vertices so that for every $u\in V(D)\setminus K$ there exists
$(u,v)\in A(D),$ where $v\in K.$ We say that a digraph $D$ is
\textit{kernel-perfect }if every nonempty induced subdigraph of $D$
has a kernel.

Given an $m$-colored digraph $D,$ we define the $k$\textit{-colored
closure}
of $D,$ denoted by $\mathfrak{C}_{k}(D),$ as the digraph such that $V(%
\mathfrak{C}_{k}(D))=V(D)$ and%
\[
A(\mathfrak{C}_{k}(D))=\{(u,v):\exists \,u\leadsto _{j}v,1\leq j\leq
k\}.
\]

\begin{remark}
Observe that every $m$-colored digraph $D$ has a $k$-colored kernel
if and only if $\mathfrak{C}_{k}(D)$ has a kernel. \label{closure}
\end{remark}

We will use the following theorem of P. Duchet \cite{Duchet}.

\begin{theorem}
If every directed cycle of a digraph $D$ has a symmetric arc, then
$D$ is kernel-perfect. \label{Duchet}
\end{theorem}

The symbol $\bigtriangleup $ will be used to denote the end of a
claim or a subclaim. We follow \cite{BJ-G} for the general
terminology on digraphs.

\section{Preliminary results}

We set $r\geq 3$ for the rest of the paper. We denote by $\mathcal{A},%
\mathcal{B},\mathcal{C},\ldots $ the partite sets of a semicomplete
multipartite digraph $D.$

\begin{lemma}
Let $D$ be an $m$-colored semicomplete $r$-partite digraph and
$x,y\in V(D).$
If there exists $x\leadsto _{k}y$ with $k\geq 4$ and there does not exist $%
y\leadsto _{k^{\prime }}x$ with $k^{\prime }\leq 4,$ then $d(x,y)\leq 2.$ \label%
{lemma-k-ge-4}
\end{lemma}

\begin{proof}
Suppose that $x\in \mathcal{A}$ and $y\in \mathcal{B}.$ Since there
does not exist $y\leadsto _{k^{\prime }}x,$ we have that $(x,y)\in
A(D).$ So, we
assume that $x,y\in \mathcal{A}$ and by contradiction, suppose that $%
d(x,y)\geq 3.$ Consider the directed path of minimum length
\[
x\longrightarrow x_{1}\longrightarrow x_{2}\longrightarrow \cdots
\longrightarrow x_{t}\longrightarrow y\text{\quad }(t\geq 2).
\]%
Therefore $x_{1}\in \mathcal{B}$ (with $\mathcal{B}\neq
\mathcal{A}$) and then $(y,x_{1})\in A(D).$ If $x_{2}\notin
\mathcal{A}$, then $(x_{2},x)\in A(D)$ (the arc $(x,x_{2})$ implies
a shorter path from $x$ to $y).$ In this case, the directed path
$y\longrightarrow x_{1}\longrightarrow x_{2}\longrightarrow x$ is a
$y\leadsto _{k^{\prime }}x$ with $k^{\prime }\leq 3,$ a
contradiction (there does not exist $y\leadsto _{k^{\prime }}x$ with
$k^{\prime }\leq 4).$ Hence $x_{2}\in \mathcal{A}$ and $t\geq 3$,
since $x_{t}\notin \mathcal{A}$ and $(x_{t},y)\in A(D).$ Recalling
that $x_{2}\in
\mathcal{A}$, we get that $x_{3}\notin \mathcal{A}$ and there exists $%
(x_{3},x)\in A(D)$ (the arc $(x,x_{3})$ implies a shorter path from $x$ to $%
y).$ We obtain that the directed path
\[
y\longrightarrow x_{1}\longrightarrow x_{2}\longrightarrow
x_{3}\longrightarrow x
\]%
is a $y\leadsto _{k^{\prime }}x$ with $k^{\prime }\leq 4,$ a
contradiction to the supposition of the lemma.
\end{proof}

\begin{lemma}
Let $D$ be an $m$-colored semicomplete $r$-partite digraph, $k=2$
(resp. $k=3 $) and $x,y\in V(D).$ If there exists $x\leadsto _{k}y$
and there does not exist $y\leadsto _{k^{\prime }}x$ with $k^{\prime
}\leq 2$ (resp. $k^{\prime }\leq 3$), then $d(x,y)\leq 4.$
\label{less-eq-4}
\end{lemma}

\begin{proof}
Suppose that $x\in \mathcal{A}$ and $y\in \mathcal{B}.$ Since there
does not exist $y\leadsto _{k^{\prime }}x,$ we have that $(x,y)\in
A(D).$ So, we
assume that $x,y\in \mathcal{A}$ and by contradiction, suppose that $%
d(x,y)\geq 5.$ Consider the directed path of minimum length
\[
x\longrightarrow x_{1}\longrightarrow x_{2}\longrightarrow \cdots
\longrightarrow x_{t}\longrightarrow y\text{\quad }(t\geq 4).
\]%
If $x_{2}\notin \mathcal{A}$, then $(y,x_{2})\in A(D)$ and
$(x_{2},x)\in A(D) $ (observe that $(x_{2},y)\in A(D)$ or
$(x,x_{2})\in A(D)$ implies a shorter path from $x$ to $y).$
Therefore the directed path $y\longrightarrow x_{2}\longrightarrow
x$ is a $y\leadsto _{k^{\prime }}x$ with $k^{\prime }\leq 2,$ a
contradiction to the supposition of the lemma. Hence $x_{2}\in
\mathcal{A}$ and so $x_{3}\notin A.$ In a similar way as done before, $%
(y,x_{3})\in A(D)$ and $(x_{3},x)\in A(D).$ It follows that the
directed path $y\longrightarrow x_{3}\longrightarrow x$ is a
$y\leadsto _{k^{\prime }}x$ with $k^{\prime }\leq 2,$ a
contradiction to the supposition of the lemma. The proof for $k=3$
follows analogously.
\end{proof}

\begin{lemma}
Let $D$ be an $m$-colored semicomplete $r$-partite digraph such that every $%
\overrightarrow{C}_{4}$ contained in $D$ is at most $2$-colored, $k=3$ and $%
x,y\in V(D).$ If there exists $x\leadsto _{k}y$ and there does not exist $%
y\leadsto _{k^{\prime }}x$ with $k^{\prime }\leq 3$, then $d(x,y)\leq 2.$ %
\label{lemma-k-eq-3}
\end{lemma}

\begin{proof}
Suppose that $x\in \mathcal{A}$ and $y\in \mathcal{B}.$ Since there
does not exist $y\leadsto _{k^{\prime }}x,$ we have that $(x,y)\in
A(D).$ So, we assume that $x,y\in \mathcal{A}$. By Lemma
\ref{less-eq-4}, $d(x,y)\leq 4.$ We consider two cases.

\textsc{Case 1.} $d(x,y)=3.$ Let $x\longrightarrow
x_{1}\longrightarrow
x_{2}\longrightarrow y$ be a directed path from $x$ to $y.$ Since $%
x_{1},x_{2}\notin \mathcal{A}$, we have that $(y,x_{1}),(x_{2},x)\in
A(D)$
and so $y\longrightarrow x_{1}\longrightarrow x_{2}\longrightarrow x$ is a $%
y\leadsto _{k^{\prime }}x$ with $k^{\prime }\leq 3,$ a contradiction
to the supposition of the lemma.

\textsc{Case 2.} $d(x,y)=4.$ Let $x\longrightarrow
x_{1}\longrightarrow
x_{2}\longrightarrow x_{3}\longrightarrow y$ be a directed path from $x$ to $%
y.$ If $x_{2}\notin \mathcal{A}$, then $(y,x_{2}),(x_{2},x)\in A(D)$ and so $%
y\longrightarrow x_{2}\longrightarrow x$ is a $y\leadsto _{k^{\prime
}}x$ with $k^{\prime }\leq 2,$ a contradiction to the supposition of
the lemma. Hence $x_{2}\in \mathcal{A}$. Notice that
$x_{1},x_{3}\notin \mathcal{A}$ and then $(y,x_{1}),(x_{3},x)\in
A(D).$ So the directed path
\[
y\longrightarrow x_{1}\longrightarrow x_{2}\longrightarrow
x_{3}\longrightarrow x
\]%
is a $y\leadsto _{k^{\prime }}x$ with $k^{\prime }=4$ (that is, a
heterochromatic directed path from $y$ to $x$), otherwise there exists a $%
y\leadsto _{k^{\prime }}x$ with $k^{\prime }\leq 3,$ a contradiction
to the supposition of the lemma. Therefore
$(y,x_{1},x_{2},x_{3},y)\cong
\overrightarrow{C}_{4}$ is at least $3$-colored, a contradiction, every $%
\overrightarrow{C}_{4}$ of $D$ is at most $2$-colored.
\end{proof}

\begin{lemma}
Let $D$ be an $m$-colored semicomplete $r$-partite digraph such that every $%
\overrightarrow{C}_{3}$ and $\overrightarrow{C}_{4}$ contained in
$D$ is monochromatic, $k=2$ and $x,y\in V(D).$ If there exists
$x\leadsto _{k}y$
and there does not exist $y\leadsto _{k^{\prime }}x$ with $k^{\prime }\leq 2$%
, then $d(x,y)\leq 2.$ \label{lemma-k-eq-2}
\end{lemma}

\begin{proof}
Suppose that $x\in \mathcal{A}$ and $y\in \mathcal{B}.$ Since there
does not exist $y\leadsto _{k^{\prime }}x,$ we have that $(x,y)\in
A(D).$ So, we assume that $x,y\in \mathcal{A}$. By Lemma
\ref{less-eq-4}, $d(x,y)\leq 4.$ We consider two cases.

\textsc{Case 1.} $d(x,y)=3.$ Let $x\longrightarrow
x_{1}\longrightarrow
x_{2}\longrightarrow y$ be a directed path from $x$ to $y.$ Since $%
x_{1},x_{2}\notin \mathcal{A}$, we have that $(y,x_{1}),(x_{2},x)\in
A(D).$ Since $(y,x_{1},x_{2},y),(x,x_{1},x_{2},x)\cong
\overrightarrow{C}_{3}$ are monochromatic, the directed path
$y\longrightarrow x_{1}\longrightarrow x_{2}\longrightarrow x$ is a
monochromatic $y\leadsto x,$ a contradiction to the supposition of
the lemma.

\textsc{Case 2.} $d(x,y)=4.$ Let $x\longrightarrow
x_{1}\longrightarrow
x_{2}\longrightarrow x_{3}\longrightarrow y$ be a directed path from $x$ to $%
y.$ If $x_{2}\notin \mathcal{A}$, then $(y,x_{2}),(x_{2},x)\in A(D)$ and so $%
y\longrightarrow x_{2}\longrightarrow x$ is a $y\leadsto _{k^{\prime
}}x$ with $k^{\prime }\leq 2,$ a contradiction to the supposition of
the lemma. Hence $x_{2}\in \mathcal{A}$. Notice that
$x_{1},x_{3}\notin \mathcal{A}$
and then $(y,x_{1}),(x_{3},x)\in A(D).$ Since $%
(y,x_{1},x_{2},x_{3},y),(x,x_{1},x_{2},x_{3},x)\cong
\overrightarrow{C}_{4}$ are monochromatic, the directed path
$y\longrightarrow x_{1}\longrightarrow x_{2}\longrightarrow
x_{3}\longrightarrow x$ is a monochromatic $y\leadsto x, $ a
contradiction to the supposition of the lemma.
\end{proof}

Analogously, we can prove the following lemma in case of
semicomplete bipartite digraphs.

\begin{lemma}
Let $D$ be an $m$-colored semicomplete bipartite digraph such that every $%
\overrightarrow{C}_{4}\upuparrows \overrightarrow{C}_{4}$ contained
in $D$ is at most $k$-colored, $k=2$ (resp. $k=3$) and $x,y\in
V(D).$ If there exists $x\leadsto _{k}y$ and there does not exist
$y\leadsto _{k^{\prime }}x$
with $k^{\prime }\leq 2$ (resp. $k^{\prime } \leq 3$), then $d(x,y)\leq 2.$ \label%
{lemma-bipart}
\end{lemma}

\section{Flowers, cycles and closed walks in the $k$-colored closure of
semicomplete $r$-partite digraphs \label{begin-proof}}

To begin with, we define the \textit{flower }$F_{s}$ with $s$ petals
as the digraph obtained by replacing every edge of the star
$K_{1,s}$ by a symmetric arc. If every edge of the complete graph
$K_{n}$ is replaced by a symmetric arc, then the resulting digraph
$D$ on $n$ vertices is symmetric semicomplete.

\begin{remark}
Let D be an $m$-colored digraph isomorphic to a
$\overrightarrow{C}_{3}$ or
a flower $F_{s}$ such that $s\geq 1.$ Then $\mathfrak{C}_{k}(F_{s})$ with $%
k\geq 2$ is a symmetric semicomplete digraph. \label{flower}
\end{remark}

This section is devoted to detail the common beginning of the proofs
of Theorems \ref{k-gr-equal-4} - \ref{bipart-k-equal-23} in the next
section.
The procedure is similar to that employed in the proof of Theorem 7 of \cite%
{GS-Ll-MB}. We include it here to make this work self-contained. In
every case, we apply Theorem \ref{Duchet} and Remark \ref{closure}
to show that every directed cycle of the $k$-colored closure
$\mathfrak{C}_{k}(D)$ of the corresponding digraph $D$ has a
symmetric arc and we proceed by contradiction.

First, we make a sketch of the following procedure in general terms.
We suppose that there exists a directed cycle $\gamma $ in
$\mathfrak{C}_{k}(D)$
without symmetric arcs and using Lemmas \ref{lemma-k-ge-4}, \ref%
{lemma-k-eq-3}, \ref{lemma-k-eq-2} and \ref{lemma-bipart} according
to each specific case, we prove that every arc of $\gamma $
corresponds to an arc or a directed path of length $2$ in the
original digraph $D.$ At this point, we consider the closed walk
$\delta ,$ subdigraph of $D,$ constructed by the concatenation of
the already mentioned arcs or directed paths of length $2$
and study its properties. In the next step, we define a closed subwalk $%
\varepsilon $ of $\delta $ satisfying some prefixed properties.
Then, we show that this subdigraph of $\delta $ exists and can be
described in a neat form.

Formally, let $D$ be an $m$-colored semicomplete $r$-partite digraph with $%
r\geq 2$. By contradiction, suppose that $\gamma
=(u_{0},u_{1},\ldots ,u_{p},u_{0})$ is a cycle in
$\mathfrak{C}_{k}(D)$ without any symmetric arc. Observe that if
$p=1,$ then $\gamma $ has a symmetric arc and we are done. So,
assume that $p\geq 2.$ Let $x$ and $y$ be two consecutive vertices
of $\gamma .$ Consider the following instances recalling that
$\gamma $ has no symmetric arcs:

\begin{enumerate}
\item[(a)] $k\geq 4.$ The conditions of Lemma \ref{lemma-k-ge-4} are
satisfied and we conclude that $d(x,y)\leq 2.$

\item[(b)] $k=3$ and every $\overrightarrow{C}_{4}$ contained in $D$ is at
most $2$-colored. The conditions of Lemma \ref{lemma-k-eq-3} are
satisfied and we conclude that $d(x,y)\leq 2.$

\item[(c)] $k=2$ and every $\overrightarrow{C}_{3}$ and $\overrightarrow{C}%
_{4}$ contained in $D$ is monochromatic. The conditions of Lemma \ref%
{lemma-k-eq-2} are satisfied and we conclude that $d(x,y)\leq 2.$

\item[(d)] $r=2,$ $k=2$ or $3$ and every $\overrightarrow{C}_{4}\upuparrows
\overrightarrow{C}_{4}$ contained in $D$ is at most $k$-colored. The
conditions of Lemma \ref{lemma-bipart} are satisfied and we conclude that $%
d(x,y)\leq 2.$
\end{enumerate}

Therefore, in any case we can assume that every arc of $\gamma $
corresponds to an arc or a directed path of length $2$ in $D.$ Let
$\delta $ be the closed directed walk defined by the concatenation
of the arcs and the directed paths of length $2$ corresponding to
the arcs of $\gamma .$

\begin{remark}

There exist at least two consecutive vertices of $\gamma $ in every
directed walk of length at least $3$ of $\delta .$ \label{rem-n}
\end{remark}

The following lemma settles two simple properties of $\delta .$

\begin{lemma}

Let $\delta $ be defined as before. Therefore

\begin{enumerate}
\item[(i)] $\delta $ contains a directed path of length at least $3$ and

\item[(ii)] there are neither $\overrightarrow{C}_{3}$ nor flowers $F_{s}$
with $s\geq 2$ in $\delta .$ \label{prop-delta}
\end{enumerate}
\end{lemma}

\begin{proof}

For the first claim, suppose that every directed path of $\delta $
has length at most $2.$ Then, either $\delta $ contains a
$\overrightarrow{C}_{3} $ or $\delta $ is isomorphic to a flower
$F_{s}$ with $s\geq 2$ and by Remarks \ref{flower} and \ref{rem-n},
$\gamma $ has a symmetric arc, a
contradiction. For the second, observe that if $\delta $ contains $%
\overrightarrow{C}_{3}$ or a flower $F_{s}$ with $s\geq 2$, then by Remark %
\ref{rem-n}, two consecutive vertices $u_{i}$ and $u_{i+1}$ of
$\gamma $
(the subindices are taken modulo $p$) belong to the vertices of a $%
\overrightarrow{C}_{3}$ or a flower $F_{s}$, respectively. So, by Remark \ref%
{flower}, there exists a symmetric arc between $u_{i}$ and $u_{i+1}$ of $%
\gamma ,$ which is a contradiction.
\end{proof}

Let $\delta =(y_{0},y_{1},...,y_{s}).$

\begin{remark}

\begin{enumerate}
\item[(i)] If there exists a flower $F_{s}$ in $\delta ,$ then $s=1.$

\item[(ii)] There are no consecutive vertices of $\gamma $ in a flower.

\item[(iii)] If there exists a subdigraph $y_{j}\longrightarrow
y_{j+1}\longrightarrow y_{j+2}\longrightarrow y_{j+3}\longrightarrow
y_{j+4}$ of $\delta ,$ where $y_{j+1}=y_{j+3}$ (that is,
$y_{j+1}\longleftrightarrow y_{j+2}$ is a flower), then $y_{j+2}\in
V(\gamma ).$ \label{rem-nplus1}
\end{enumerate}
\end{remark}

Notice that if $\delta =\gamma ,$ as we will see, the same argument
of the proof will work even easier.

Observe that there exist $y_{i_{0}},y_{i_{1}},...y_{i_{p}}\in
V(\delta )$ such that $i_{j}<i_{j+1}$ and $u_{l}=y_{j_{l}},$ where
$0\leq l\leq p.$

We define $\varepsilon =(y_{i},y_{i+1},\ldots ,y_{i+l})$ of minimum length ($%
0\leq i\leq s$ and the indices are taken modulo $s+1)$ such that

\begin{enumerate}
\item[(i)] $y_{i}=y_{i+l},$ $l\geq 3,$

\item[(ii)] $y_{i}\neq y_{t}$ for $i+1\leq t\leq i+l-1,$

\item[(iii)] if $y_{q}=y_{r},$ then $r=q+2$ ($i+1\leq q,r\leq i+l-1$)$,$

\item[(iv)] there exist $y_{i_{1}},y_{i_{2}},\ldots ,y_{i_{k+1}}\in
V(\varepsilon )$ such that $y_{i_{1}}=u_{j},$ $y_{i_{2}}=u_{j+1},$
$\ldots ,$ $y_{i_{k+1}}=u_{j+k}$ with $k\geq 1,$ and

\item[(v)] $y_{i+1}\neq y_{i+l-1}.$
\end{enumerate}

\begin{lemma}

There exists $\varepsilon $ a closed subwalk of $\delta .$
\end{lemma}

\begin{proof}

Since $\delta $ is a closed walk, $p\geq 2$ and using Lemma \ref{prop-delta}%
(i), condition (i) is satisfied. For (ii), if there exists $t<l$ such that $%
y_{i}=y_{i+t},$ then, by the minimality of $\varepsilon ,$ $t=2$ and
$l-t=2$ and therefore $(i+l)-(i+t)=2$. By (i), we have that
$y_{i}=y_{i+t}=y_{i+l}$ and so $l=4$. We obtain that
\[
y_{i+1}\longleftrightarrow y_{i}=y_{i+2}=y_{i+4}\longleftrightarrow
y_{i+3},
\]
which is a flower $F_{2}$ in $\delta $, a contradiction to Lemma \ref%
{prop-delta}(ii). Condition (iii) follows from the minimality of $%
\varepsilon $ and condition (iv) is immediate from the definition of
$\delta $ and the fact that $l\geq 3.$ If $y_{i+1}=y_{i+l-1},$ then
by (iii), $l=4$
and hence%
\[
y_{i+4}=y_{i}\longleftrightarrow y_{i+1}=y_{i+3}\longleftrightarrow
y_{i+2}
\]%
which is a flower $F_{2}$ in $\delta $, a contradiction to Lemma \ref%
{prop-delta}(ii). Condition (v) follows.
\end{proof}

Since $\delta $ is not a flower itself by supposition, we can
establish the structure of $\varepsilon $ with precision.

\begin{corollary}

The closed subwalk $\varepsilon $ of $\delta $ is a directed cycle
of length at least $3$ with perhaps symmetric arcs attached to some
vertices (maybe none) of the cycle for which the exterior endpoints
are vertices of $\gamma . $
\end{corollary}

An example of $\varepsilon $ is depicted in \cite{GS-Ll-MB}.

For the sake of a clearer exposition of the forthcoming proofs in
the next section, let us rename $\varepsilon =(y_{0},y_{1},\ldots
,y_{l}).$ By (v) of the definition of $\varepsilon $, we have that
$y_{1}\neq y_{l-1}$ and by
(iv), there exist consecutive $u_{0},u_{1},\ldots ,u_{k}\in V(\gamma )$ in $%
\varepsilon $ with $k\geq 1.$ Notice that $u_{0}$ and $u_{k}$ could
not be consecutive vertices of $\gamma $ and similarly,
$(y_{l-1},y_{0})\in A(\varepsilon )$ could not be an arc of $\gamma
.$ Let $u_{1}=y_{i}$ be the second vertex of $\gamma $ from $y_{0}.$
Observe that $1\leq i\leq 3$ by the
definition of $\varepsilon $ and either $u_{0}=y_{0}$ $(1\leq i\leq 2)$ or $%
u_{0}=y_{1}$ $(2\leq i\leq 3).$

Let us suppose that there exists $(u_{1},y_{0})\in A(D),$ then

\begin{enumerate}
\item[(i)] if $u_{0}=y_{0},$ then $(u_{1},u_{0})\in A(D)$ and $\gamma $ has
a symmetric arc between $u_{0}$ and $u_{1},$ a contradiction,

\item[(ii)] if $u_{0}=y_{1},$ then we have that $u_{1}\longrightarrow
y_{0}\longrightarrow y_{1}=u_{0}$ and we arrive at the contradiction
of (i).
\end{enumerate}

Therefore, without loss of generality we can assume that
\begin{equation}
(u_{1},y_{0})\notin A(D).  \tag{$\nabla $}
\end{equation}

\section{Main theorems}

First, we recall that by supposition $\gamma =(u_{0},u_{1},\ldots
,u_{p},u_{0})$ is a cycle in $\mathfrak{C}_{k}(D)$ without any
symmetric arc
and $\varepsilon =(y_{0},y_{1},\ldots ,y_{l})$ is a closed subwalk of $%
\delta .$ The beginning of every proof of the following theorems are
the arguments stated in Section \ref{begin-proof}.

\begin{theorem}

Let $D$ be an $m$-colored semicomplete $r$-partite digraph and
$k\geq 4.$ Then $D$ has a $k$-colored kernel. \label{k-gr-equal-4}
\end{theorem}

\begin{proof}

In this case, we use instance (a) to assume that every arc of
$\gamma $ corresponds to an arc or a directed path of length $2$ in
$D.$

\textbf{Claim. }$(u_{1},y_{j})\in A(D)$ for some $l-2\leq j\leq
l-1.$

\textit{Proof of the claim.} To prove the claim, suppose by
contradiction that $q$ is the maximum index such that
$(u_{1},y_{q})\in A(D)$ with $q\leq l-3.$ Consider the directed path
\[
u_{1}\longrightarrow y_{q}\longrightarrow y_{q+1}\longrightarrow
y_{q+2}\longrightarrow y_{q+3}
\]%
(observe that $q+3\leq l$). First, we will show that $u_{1}$ and
$y_{q+3}$ belong to the same part of $D.$ If $(u_{1},y_{q+3})\in
A(D),$ then $q+3=l$
because $q$ is maximum and $y_{q+3}=y_{0},$ a contradiction to ($%
\triangledown $). If $(y_{q+3},u_{1})\in A(D),$ then we have the
directed cycle $(u_{1},y_{q},y_{q+1},y_{q+2},y_{q+3},u_{1}).$

If $y_{q}=u_{t}$ for some $2\leq t\leq p,$ then either $y_{q+1}=u_{t+1}$ or $%
y_{q+2}=u_{t+1}.$ Therefore there exists $u_{t+1}\leadsto _{k}u_{t}$ with $%
k\leq 4,$ a contradiction, $\gamma $ has a symmetric arc between
$u_{t}$ and
$u_{t+1}.$ Analogously, if $y_{q+1}=u_{t}$ and either $y_{q+2}=u_{t+1}$ or $%
y_{q+3}=u_{t+1},$ then we arrive to the same contradiction as
before. Finally, if $y_{q+1}=u_{0},$ then there exists
$u_{1}\leadsto _{k}u_{0}$
with $k\leq 4,$ a contradiction, $\gamma $ has a symmetric arc between $%
u_{0} $ and $u_{1}.$ We conclude that $u_{1},y_{q+3}\in \mathcal{A}$
(the
same part of the semicomplete $r$-partite digraph $D$). As a consequence, $%
y_{q},y_{q+2}\notin \mathcal{A}$ and there exists
$(y_{q+2},u_{1})\in A(D)$ by the maximality of $q.$

We obtain the directed cycle $\overrightarrow{C}_{4}\cong
(y_{q},y_{q+1},y_{q+2},u_{1},y_{q})$ in which there are no two
consecutive vertices of $\gamma ,$ otherwise $\gamma $ has a
symmetric arc, a contradiction. Hence, $y_{q+1}=u_{t}$ and
$y_{q+3}\in V(\gamma ).$ Since
there exists $u_{1}\leadsto _{k}y_{q+3}$ with $k=4,$ we have that $%
y_{q+3}\neq u_{0},$ otherwise $\gamma $ has a symmetric arc between
$u_{0}$ and $u_{1}.$ Then $y_{q+3}=u_{t+1}$ and thus $q+3<l$ and we
consider the
extended directed path%
\[
u_{1}\longrightarrow y_{q}\longrightarrow
y_{q+1}=u_{t}\longrightarrow y_{q+2}\longrightarrow
y_{q+3}=u_{t+1}\longrightarrow y_{q+4}.
\]%
Recall that $u_{1},y_{q+3}\in \mathcal{A}$ and then $y_{q+4}\notin \mathcal{A%
}$. By the maximality of $q$ and since $(u_{1},y_{0})\notin A(D),$
there
exists $(y_{q+4},u_{1})\in A(D).$ We obtain the directed path%
\[
u_{t+1}=y_{q+3}\longrightarrow y_{q+4}\longrightarrow
u_{1}\longrightarrow y_{q}\longrightarrow y_{q+1}=u_{t},
\]%
a contradiction, there exists a symmetric arc between $u_{t}$ and
$u_{t+1}$ in $\gamma .$ The claim is proved. \hfill $\bigtriangleup
$

We conclude the proof of the theorem applying the Claim. In the
worst case, we have that $q=l-2$ and $y_{1}=u_{0}$. We obtain the
directed cycle
\[
\overrightarrow{C}_{4}\cong
(u_{1},y_{q}=y_{l-2},y_{l-1},y_{0},y_{1}=u_{0},u_{1})
\]%
and there exists $u_{1}\leadsto _{k}u_{0}$ with $k=4,$ a contradiction, $%
\gamma $ has a symmetric arc between $u_{0}$ and $u_{1}.$ In any
other case, there exists $u_{1}\leadsto _{k}u_{0}$ with $k\leq 4$
and it yields the same contradiction as before.
\end{proof}

\begin{theorem}

Let $D$ be an $m$-colored semicomplete $r$-partite digraph and
$k=3.$ If every $\overrightarrow{C}_{4}$ contained in $D$ is at most
$2$-colored and,
either every $\overrightarrow{C}_{5}$ contained in $D$ is at most $3$%
-colored or every $\overrightarrow{C}_{3}\uparrow
\overrightarrow{C}_{3}$
contained in $D$ is at most $2$-colored, then $D$ has a $3$-colored kernel. %
\label{k-equal-3}
\end{theorem}

\begin{proof}

In this case, we use instance (b) to assume that every arc of
$\gamma $ corresponds to an arc or a directed path of length $2$ in
$D.$

\textbf{Claim. }$(u_{1},y_{j})\in A(D)$ for some $l-2\leq j\leq
l-1.$

\textit{Proof of the claim.} To prove the claim, suppose by
contradiction that $q$ is the maximum index such that
$(u_{1},y_{q})\in A(D)$ with $q\leq l-3.$ Consider the directed path
\[
u_{1}\longrightarrow y_{q}\longrightarrow y_{q+1}\longrightarrow
y_{q+2}.
\]%
First, we will show that $u_{1}$ and $y_{q+2}$ belong to the same part of $%
D. $ Observe that $(u_{1},y_{q+2})\in A(D)$ is impossible by the
choice of $q $ and since $q+2<l.$ Therefore, we suppose that
$(y_{q+2},u_{1})\in A(D).$ If there exist $u_{t},u_{t+1}\in V(\gamma
)$ (indices are taken modulo $p+1),
$ such that $u_{t},u_{t+1}\in \{y_{q},y_{q+1},y_{q+2}\},$ then there exists $%
u_{t+1}\leadsto _{k}u_{t}$ with $k\leq 3,$ a contradiction, $\gamma
$ has a
symmetric arc between $u_{t}$ and $u_{t+1}.$ Hence $y_{q+1}=u_{t}$ and $%
y_{q},y_{q+2}\notin V(\gamma ).$ Since the directed cycle $\overrightarrow{C}%
_{4}\cong (u_{1},y_{q},y_{q+1},y_{q+2},u_{1})$ is at most
$2$-colored, the directed path
\[
u_{1}\longrightarrow y_{q}\longrightarrow y_{q+1}\longrightarrow
y_{q+2}\longrightarrow y_{q+3}
\]%
is at most $3$-colored and then $y_{q+3}\neq u_{0}$ and
$y_{q+3}=u_{t+1}$ (in virtue of the definition of $\varepsilon ).$
Moreover, $q+3<l.$ Let us suppose that there exists an arc between
$y_{q+3}$ and $u_{1}.$ By the
maximality of $q,$ we have that $(y_{q+3},u_{1})\in A(D).$ Then%
\[
y_{q+3}=u_{t+1}\longrightarrow u_{1}\longrightarrow
y_{q}\longrightarrow y_{q+1}=u_{t}
\]%
is an at most $3$-colored directed path, a contradiction, $\gamma $
has a
symmetric arc between $u_{t}$ and $u_{t+1}.$ In consequence, $%
u_{1},y_{q+3}\in \mathcal{A}$ (a same part of $D)$, $y_{q}\notin
\mathcal{A}$ and there exists an arc between $y_{q}$ and $y_{q+3}.$

If $(y_{q+3},y_{q})\in A(D),$ then the directed path
\[
y_{q+3}=u_{t+1}\longrightarrow y_{q}\longrightarrow y_{q+1}=u_{t}
\]%
is an at most $3$-colored directed path, a contradiction, $\gamma $
has a symmetric arc between $u_{t}$ and $u_{t+1}.$ Thus,
$(y_{q},y_{q+3})\in A(D)$
and let us consider the extended directed path%
\[
u_{1}\longrightarrow y_{q}\longrightarrow
y_{q+1}=u_{t}\longrightarrow y_{q+2}\longrightarrow
y_{q+3}=u_{t+1}\longrightarrow y_{q+4},
\]%
where $q+4\leq l$ and furthermore, $y_{q+4}\notin \mathcal{A}$.
There exists the arc between $u_{1}$ and $y_{q+4}.$ If
$(u_{1},y_{q+4})\in A(D),$ then by the maximality of $q,$
$y_{q+4}=y_{l}=y_{0}$ and we obtain a contradiction to the
assumption $(\nabla ).$ So, $(y_{q+4},u_{1})\in A(D).$

Recalling that $(y_{q+3},u_{1}),(y_{q},y_{q+3}),(y_{q+4},u_{1})\in
A(D),$ we have the directed cycles
\begin{eqnarray}
\overrightarrow{C}_{4} &\cong &(u_{1},y_{q},y_{q+1}=u_{t},y_{q+2},u_{1})%
\text{ and }  \nonumber \\
\overrightarrow{C}_{4} &\cong
&(y_{q+3}=u_{t+1},y_{q+4},u_{1},y_{q},y_{q+3}) \label{cycles}
\end{eqnarray}%
which are at most $2$-colored by the condition of the theorem. Then
the directed path
\[
u_{t+1}=y_{q+3}\longrightarrow y_{q+4}\longrightarrow
u_{1}\longrightarrow y_{q}\longrightarrow y_{q+1}=u_{t}
\]%
is at most $3$-colored given that the directed cycles of
(\ref{cycles}) have the common arc $(u_{1},y_{q}).$ We have a
contradiction, $\gamma $ has a
symmetric arc between $u_{t}$ and $u_{t+1}.$ We conclude that $%
u_{1},y_{q+2}\in \mathcal{A}$ (a same part of $D).$

As a consequence, $y_{q},y_{q+1},y_{q+3}\notin \mathcal{A}$. The
maximality
of $q$ implies that $(y_{q+1},u_{1})\in A(D).$ If there exists $%
(u_{1},y_{q+3})\in A(D),$ then by the maximality of $q,$ $%
y_{q+3}=y_{l}=y_{0},$ a contradiction to $(\nabla ).$ So, there
exists $(y_{q+3},u_{1})\in A(D).$

If $(y_{q},y_{q+2})\in A(D),$ then there exists the directed cycle
\[
\overrightarrow{C}_{4}\cong (u_{1},y_{q},y_{q+2},y_{q+3},u_{1})
\]%
which is at most $2$-colored and therefore there exists
$u_{t+1}\leadsto _{k}u_{t}$ where $k\leq 3$ and with
$u_{t},u_{t+1}\in \{y_{q},y_{q+1},y_{q+2},y_{q+3}\}$, a
contradiction, $\gamma $ has a symmetric arc between $u_{t}$ and
$u_{t+1}.$ Hence $(y_{q+2},y_{q})\in A(D).$ \hfill $\bigtriangleup $

In brief, we have that
$(y_{q+1},u_{1}),(y_{q+3},u_{1}),(y_{q+2},y_{q})\in A(D)$ in the
directed path
\[
u_{1}\longrightarrow y_{q}\longrightarrow y_{q+1}\longrightarrow
y_{q+2}\longrightarrow y_{q+3}.
\]%
Observe that $u_{t}\in \{y_{q},y_{q+1}\}.$ If $y_{q+2}=u_{t+1}$, then $%
u_{t+1}=y_{q+2}\longrightarrow y_{q}\longrightarrow y_{q+1}$ is at most $3$%
-colored, a contradiction, $\gamma $ has a symmetric arc between
$u_{t}$ and
$u_{t+1}.$ We conclude that $y_{q+1}=u_{t}$ and either $y_{q+3}=u_{t+1}$ or $%
y_{q+3}=u_{0}.$ If $y_{q+3}=u_{t+1},$ then%
\[
u_{t+1}=y_{q+3}\longrightarrow u_{1}\longrightarrow
y_{q}\longrightarrow y_{q+1}
\]%
is at most $3$-colored, a contradiction, $\gamma $ has a symmetric
arc between $u_{t}$ and $u_{t+1}.$ Thus $y_{q+3}=u_{0}.$

By condition of the theorem, every $\overrightarrow{C}_{5}$
contained in $D$ is at most $3$-colored or every
$\overrightarrow{C}_{3}\uparrow
\overrightarrow{C}_{3}$ contained in $D$ is at most $2$-colored. If every $%
\overrightarrow{C}_{5}$ is at most $3$-colored, then
\[
\overrightarrow{C}_{5}\cong
(u_{1},y_{q},y_{q+1},y_{q+2},y_{q+3}=u_{0},u_{1})
\]%
is at most $3$-colored and consequently, there exists $u_{1}\leadsto
_{k}u_{0}$ at most $3$-colored, a contradiction, $\gamma $ has a
symmetric arc between $u_{0}$ and $u_{1}.$

If every $\overrightarrow{C}_{3}\uparrow \overrightarrow{C}_{3}$ is at most $%
2$-colored then the $\overrightarrow{C}_{3}\uparrow
\overrightarrow{C}_{3}$ induced by $\{u_{1},y_{q},y_{q+1},y_{q+2}\}$
is at most $2$-colored and there exists $u_{1}\leadsto _{k}u_{0}$ at
most $3$-colored, a contradiction, $\gamma $ has a symmetric arc
between $u_{0}$ and $u_{1}.$

The claim is proved. \hfill $\bigtriangleup $

To finish the proof of the theorem, we apply the Claim and consider
two cases:

\textsc{Case 1. }$(u_{1},y_{l-1})\in A(D).$ In this case the directed path $%
u_{1}\longrightarrow y_{l-1}\longrightarrow y_{0}\longrightarrow
y_{1}$ is at most $3$-colored and we know that $u_{0}\in
\{y_{0},y_{1}\}.$ We arrive
to a similar contradiction as before, $\gamma $ has a symmetric arc between $%
u_{0}$ and $u_{1}.$

\textsc{Case 2. }$(u_{1},y_{l-2})\in A(D).$ Observe that $y_{0}\neq
u_{0}$ and $y_{1}=u_{0},$ otherwise
\[
u_{1}\longrightarrow y_{l-2}\longrightarrow y_{l-1}\longrightarrow
y_{0}=u_{0}
\]%
is at most $3$-colored, an we have the contradiction of Case 1 once
more. So, we have the directed path
\[
u_{1}\longrightarrow y_{l-2}\longrightarrow y_{l-1}\longrightarrow
y_{0}\longrightarrow y_{1}=u_{0}.
\]%
By assumption $(\nabla ),$ the arc $(y_{0},u_{1})$ could belong to
$A(D).$
If that is the case, then the directed cycle%
\[
\overrightarrow{C}_{4}\cong (u_{1},y_{l-2},y_{l-1},y_{0},u_{1})
\]%
is at most $2$-colored by the condition of the theorem and therefore, $%
u_{1}\leadsto _{k}u_{0}$ at most $3$-colored, a contradiction,
$\gamma $ has
a symmetric arc between $u_{0}$ and $u_{1}.$ So, we suppose that $%
(y_{0},u_{1})\notin A(D).$ By $(\nabla )$, we have that
$(u_{1},y_{0})\notin
A(D)$ and then $u_{1},y_{0}\in \mathcal{A}$ and $u_{0}=y_{1}\notin \mathcal{A%
}$. Consequently, there exists an arc between $y_{1}=u_{0}$ and
$u_{1}.$ It is clear that $(u_{0},u_{1})\in A(D)$ (otherwise we have
a contradiction).

If every $\overrightarrow{C}_{5}$ is at most $3$-colored, then
\[
\overrightarrow{C}_{5}\cong
(u_{1},y_{l-2},y_{l-1},y_{0},y_{1}=u_{0},u_{1})
\]%
is at most $3$-colored and we arrive to a similar contradiction as
shown before. Hence, we can suppose that there exists a
$\overrightarrow{C}_{5}$
at least $4$-colored and thus we assume the condition that every $%
\overrightarrow{C}_{3}\uparrow \overrightarrow{C}_{3}$ is at most $2$%
-colored in $D.$ Notice that $y_{l-1}\in \mathcal{B\neq A}$ and then
there exists an arc between $y_{l-1}$ and $u_{1}.$ Since $l-2$ is
the maximum index such that $(u_{1},y_{l-2})\in A(D),$ we have that
$(y_{l-1},u_{1})\in A(D).$ Also $y_{l-2}\in \mathcal{C\notin
\{A},\mathcal{B\}}$ and then there exists an arc between $y_{0}$ and
$y_{l-2}$. If $(y_{0},y_{l-2})\in A(D),$
then the $\overrightarrow{C}_{3}\uparrow \overrightarrow{C}_{3}$ induced by $%
\{u_{1},y_{l-2},y_{l-1},y_{0}\}$ is at most $2$-colored. Hence there exists $%
u_{1}\leadsto _{k}u_{0}$ at most $3$-colored, a contradiction,
$\gamma $ has a symmetric arc between $u_{0}$ and $u_{1}.$ If
$(y_{l-2},y_{0})\in A(D),$ then
\[
u_{1}\longrightarrow y_{l-2}\longrightarrow y_{0}\longrightarrow
y_{1}=u_{0}
\]%
is at most $3$-colored, a contradiction, $\gamma $ has a symmetric
arc between $u_{0}$ and $u_{1}.$

The theorem is proved.
\end{proof}

\begin{theorem}

Let $D$ be an $m$-colored semicomplete $r$-partite digraph and
$k=2.$ If every $\overrightarrow{C}_{3}$ and
$\overrightarrow{C}_{4}$ contained in $D$ is monochromatic, then $D$
has a $2$-colored kernel. \label{k-equal-2}
\end{theorem}

\begin{proof}

In this case, we use instance (c) to assume that every arc of
$\gamma $ corresponds to an arc or a directed path of length $2$ in
$D.$

\textbf{Claim 1. }$(u_{1},y_{j})\in A(D)$ for some $l-2\leq j\leq
l-1.$

\textit{Proof of the Claim 1.} To prove the claim, suppose by
contradiction that $q$ is the maximum index such that
$(u_{1},y_{q})\in A(D)$ with $q\leq l-3.$ Without loss of
generality, suppose that $u_{1}\in \mathcal{A}$. We will need the
following three subclaims.

\textbf{Subclaim 1. }$y_{q+1}\in V(\gamma ).$

\textit{Proof of the Subclaim 1. }By contradiction, suppose that $%
y_{q+1}\notin V(\gamma )$ and consider the directed path%
\[
u_{1}\longrightarrow y_{q}\longrightarrow y_{q+1}\longrightarrow
y_{q+2}
\]%
($q+3\leq l$). Then $y_{q},y_{q+2}\in V(\gamma )$ and without loss
of generality, we can suppose that $y_{q}=u_{t}$ and
$y_{q+1}=u_{t+1}$ for some
$2\leq t\leq p-1$ and $u_{t+1}\neq u_{0}$ (otherwise there exists $%
u_{1}\leadsto _{k}u_{0}$ with $k\leq 2,$ a contradiction, $\gamma $
has a
symmetric arc between $u_{0}$ and $u_{1}).$ Observe that if $%
y_{q+2}=u_{t+1}\notin \mathcal{A}$, then there exists an arc between
$u_{1}$ and $y_{q+2}.$ By the maximality of $q,$ we have that
$(y_{q+2},u_{1})\in A(D)$ and then $u_{t}$ and $u_{t+1}$ are
contained in the monochromatic cycle $\overrightarrow{C}_{4}\cong
(u_{1},y_{q},y_{q+1},y_{q+2},u_{1})$ (by hypothesis). So, there
exists a monochromatic $u_{t+1}\leadsto u_{t},$ a contradiction,
$\gamma $ has a symmetric arc between $u_{t}$ and $u_{t+1}.$
Hence, $y_{q+2}=u_{t+1}\in \mathcal{A}$. We know that $y_{q}\notin \mathcal{A%
}$. Let us suppose that $y_{q}\in \mathcal{B}\neq \mathcal{A}$.
There exists an arc between $y_{q}$ and $y_{q+2}.$ If
$(y_{q+2},y_{q})=(u_{t+1},u_{t})\in
A(D),$ then we arrive to the same contradiction as before. Therefore $%
(y_{q},y_{q+2})=(u_{t},u_{t+1})\in A(D)$. As a consequence $y_{q+1}$
does not exist in $\varepsilon $ by the definition of $\gamma ,$ a
contradiction. \hfill $\bigtriangleup $

\textbf{Subclaim 2. }$y_{q+2}\notin V(\gamma ).$

\textit{Proof of the Subclaim 2. }By contradiction, suppose that
$y_{q+2}\in V(\gamma )$ and hence $y_{q+2}=u_{t+1}$ because $q+2\leq
l-1.$ So we have
the directed path%
\[
u_{1}\longrightarrow y_{q}\longrightarrow
y_{q+1}=u_{t}\longrightarrow y_{q+2}=u_{t+1}.
\]%
If $y_{q+2}=u_{t+1}\notin \mathcal{A}$, there exists the arc $%
(y_{q+2}=u_{t+1},u_{1})\in A(D)$ by the maximality of $q.$ But then
$u_{t}$ and $u_{t+1}$ are contained in a monochromatic cycle
\[
\overrightarrow{C}_{4}\cong
(y_{q+2}=u_{t+1},u_{1},y_{q},y_{q+1}=u_{t},y_{q+2}=u_{t+1}),
\]%
a contradiction, there exists a monochromatic $u_{t+1}\leadsto
u_{t}$ and a
symmetric arc between $u_{t}$ and $u_{t+1}$ in $\gamma .$ Therefore, $%
y_{q+2}=u_{t+1}\in \mathcal{A}$ and there exists an arc between $y_{q}$ and $%
y_{q+2}=u_{t+1}$ (recall that $y_{q}\notin \mathcal{A}$). If $%
(y_{q+2}=u_{t+1},y_{q})\in A(D),$ then $%
(y_{q+1}=u_{t},y_{q+2}=u_{t+1},y_{q},y_{q+1}=u_{t})$ is a monochromatic $%
\overrightarrow{C}_{3}$ by hypothesis and there exists a monochromatic $%
u_{t+1}\leadsto u_{t}$ and we get the same contradiction. Hence $%
(y_{q},y_{q+2}=u_{t+1})\in A(D).$ Consider the extended directed
path
\[
u_{1}\longrightarrow y_{q}\longrightarrow
y_{q+1}=u_{t}\longrightarrow y_{q+2}=u_{t+1}\longrightarrow y_{q+3},
\]%
where $y_{q+3}\notin \mathcal{A}$ (since $y_{q+2}=u_{t+1}\in
\mathcal{A}$).
Then, there exists an arc between $u_{1}$ and $y_{q+3}.$ If $%
(u_{1},y_{q+3})\in A(D),$ then $y_{q+3}=y_{l}=y_{0},$ a contradiction to $%
(\nabla ).$ Thus, $(y_{q+3},u_{1})\in A(D).$ Recall that $%
(y_{q},y_{q+2}=u_{t+1})\in A(D).$ Hence,
\[
\overrightarrow{C}_{4}\cong
(y_{q+2}=u_{t+1},y_{q+3},u_{1},y_{q},y_{q+2}=u_{t+1})
\]%
is monochromatic by hypothesis. So, there exists $u_{t+1}\leadsto
_{k}u_{t}$
with $k\leq 2,$ a contradiction, $\gamma $ has a symmetric arc between $%
u_{t} $ and $u_{t+1}.$ \hfill $\bigtriangleup $

As a consequence of Subclaim 2, we have that $y_{q+3}\in V(\gamma
).$

\textbf{Subclaim 3. }$y_{q+3}\in \mathcal{A}$.

\textit{Proof of the Subclaim 3. }By contradiction, suppose that $%
y_{q+3}\notin \mathcal{A}$. By Subclaim 1, we can suppose that $%
y_{q+1}=u_{t}\in V(\gamma ).$ Consider the directed path%
\[
u_{1}\longrightarrow y_{q}\longrightarrow
y_{q+1}=u_{t}\longrightarrow y_{q+2}\longrightarrow y_{q+3}
\]%
Then there exists an arc between $u_{1}$ and $y_{q+3}$ (recall that $%
u_{1}\in \mathcal{A}$). By the maximality of $q,$ we have that $%
(y_{q+3},u_{1})\in A(D).$ By Subclaim 2, $y_{q+2}\notin V(\gamma )$
and thus $y_{q+3}\in V(\gamma ).$ We consider two cases:

\textsc{Case 1.} $y_{q+2}\notin \mathcal{A}$. By the maximality of
$q,$
there exists $(y_{q+2},u_{1})\in A(D)$ and the directed cycle%
\[
\overrightarrow{C}_{4}\cong
(y_{q+2},u_{1},y_{q},y_{q+1}=u_{t},y_{q+2})
\]%
is monochromatic by hypothesis. If $y_{q+3}=u_{0},$ then there exists $%
u_{1}\leadsto _{k}u_{0}$ with $k\leq 2,$ a contradiction, $\gamma $
has a symmetric arc between $u_{0}$ and $u_{1}.$ If
$y_{q+3}=u_{t+1},$ then here exists $u_{t+1}\leadsto _{k}u_{t}$ with
$k\leq 2,$ a contradiction, $\gamma $ has a symmetric arc between
$u_{t}$ and $u_{t+1}.$

\textsc{Case 2.} $y_{q+2}\in \mathcal{A}$. Then $y_{q+1}\notin
\mathcal{A}$ and since $y_{q}\in \mathcal{B}$ (see the proof of
Subclaim 1), we have that
$y_{q+1}\notin \mathcal{B}$. Without loss of generality, suppose that $%
y_{q+1}\in \mathcal{C}$. By the maximality of $q,$ there exists $%
(y_{q+1},u_{1})\in A(D)$ and therefore the directed cycle
\begin{equation}
\overrightarrow{C}_{3}\cong (u_{1},y_{q},y_{q+1}=u_{t},u_{1})
\label{3-cycle}
\end{equation}%
is monochromatic. On the other hand, there exists an arc between
$y_{q}$ and
$y_{q+2}.$ If $(y_{q+2},y_{q})\in A(D),$ then the directed cycle $%
\overrightarrow{C}_{3}\cong (y_{q+2},y_{q},y_{q+1}=u_{t},y_{q+2})$
is
monochromatic and has the same color of the $\overrightarrow{C}_{3}$ of (\ref%
{3-cycle}) because they share the arc $(y_{q},y_{q+1}=u_{t})\in A(D).$ If $%
y_{q+3}=u_{0},$ then there exists $u_{1}\leadsto _{k}u_{0}$ with
$k\leq 2,$ a contradiction, $\gamma $ has a symmetric arc between
$u_{0}$ and $u_{1}.$
If $y_{q+3}=u_{t+1},$ then there exists $u_{t+1}\leadsto _{k}u_{t}$ with $%
k\leq 2,$ a contradiction, $\gamma $ has a symmetric arc between
$u_{t}$ and $u_{t+1}.$ So, we conclude that there exists
$(y_{q},y_{q+2})\in A(D).$ In
this case, the directed cycle%
\[
\overrightarrow{C}_{4}\cong (u_{1},y_{q},y_{q+2},y_{q+3},u_{1})
\]%
is monochromatic and of the same color as the $\overrightarrow{C}_{3}$ of (%
\ref{3-cycle}) because they share the arc $(u_{1},y_{q})\in A(D).$
Analogously, if $y_{q+3}=u_{0}$ or $y_{q+3}=u_{t+1},$ we arrive to
the same contradiction, $\gamma $ has a symmetric arc.

The subclaim follows. \hfill $\bigtriangleup $

Continuing with the proof of Claim 1, we have the following directed path%
\[
u_{1}\longrightarrow y_{q}\longrightarrow
y_{q+1}=u_{t}\longrightarrow y_{q+2}\longrightarrow y_{q+3},
\]%
where $u_{1},y_{q+3}\in \mathcal{A}$, $y_{q}\in \mathcal{B}$,
$y_{q+2}\notin \mathcal{A}$, $y_{q+2}\notin V(\gamma )$ and
$y_{q+3}\in V(\gamma )$ using Subclaims 1-3. By the maximality of
$q,$ there exists $(y_{q+2},u_{1})\in
A(D)$ creating the monochromatic directed cycle%
\begin{equation}
\overrightarrow{C}_{4}\cong
(y_{q+2},u_{1},y_{q},y_{q+1}=u_{t},y_{q+2}). \label{4-cycle}
\end{equation}%
Hence, there exists $u_{1}\leadsto _{k}y_{q+3}$ with $k\leq 2$ and
therefore, $y_{q+3}\neq u_{0}$ (otherwise, $\gamma $ has a symmetric
arc between $u_{0}$ and $u_{1}$) and then $y_{q+3}=u_{t+1}$ with
$q+3<l.$ Consider the extended directed path
\[
u_{1}\longrightarrow y_{q}\longrightarrow
y_{q+1}=u_{t}\longrightarrow y_{q+2}\longrightarrow
y_{q+3}=u_{t+1}\longrightarrow y_{q+4},
\]%
where $y_{q+4}\notin \mathcal{A}$ (since $y_{q+3}\in \mathcal{A}$).
Therefore there exists an arc between $u_{1}$ and $y_{q+4}.$ If $%
(u_{1},y_{q+4})\in A(D),$ then by the maximality of $q,$ we have that $%
y_{q+4}=y_{0},$ a contradiction to $(\nabla ).$ So there exists $%
(y_{q+4},u_{1})\in A(D).$

On the other hand, since $y_{q+3}=u_{t+1}\in \mathcal{A}$ and
$y_{q}\in
\mathcal{B}$, there exists an arc between $y_{q}$ and $y_{q+3}.$ If $%
(y_{q+3}=u_{t+1},y_{q})\in A(D),$ then $u_{t}$ and $u_{t+1}$ belong
to the monochromatic
\[
\overrightarrow{C}_{4}\cong
(y_{q},y_{q+1}=u_{t},y_{q+2},y_{q+3}=u_{t+1},y_{q})
\]%
and thus there exists a monochromatic $u_{t+1}\leadsto u_{t},$ a
contradiction, $\gamma $ has a symmetric arc between $u_{t}$ and
$u_{t+1}.$ Hence $(y_{q},y_{q+3}=u_{t+1})\in A(D)$ and we obtain the
monochromatic
directed cycle%
\[
\overrightarrow{C}_{4}\cong
(y_{q},y_{q+3}=u_{t+1},y_{q+4},u_{1},y_{q})
\]%
of the same color of the $\overrightarrow{C}_{4}$ of (\ref{4-cycle})
because they share the arc $(u_{1},y_{q})\in A(D).$ Thus, there
exists the
monochromatic $u_{t+1}=y_{q+3}\leadsto y_{q+1}=u_{t},$ a contradiction, $%
\gamma $ has a symmetric arc between $u_{t}$ and $u_{t+1}.$

Claim 1 is proved. \hfill $\bigtriangleup $

\textbf{Claim 2. }$y_{0}\in \mathcal{A}$.

\textit{Proof of Claim 2. }By contradiction, let us suppose that $%
y_{0}\notin \mathcal{A}$. Then there exists an arc between $u_{1}$ and $%
y_{0}.$ By $(\nabla ),$ there exists $(y_{0},u_{1})\in A(D).$ So
the directed cycle%
\[
\overrightarrow{C}_{4}\cong (y_{0},u_{1},y_{l-2},y_{l-1},y_{0})\text{ or }%
\overrightarrow{C}_{3}\cong (y_{0},u_{1},y_{l-1},y_{0})
\]%
is monochromatic and hence there exists $u_{1}\leadsto _{k}u_{0}$ with $%
k\leq 2$ (recall that $y_{0}=u_{0}$ or $y_{1}=u_{0}$). We arrive to
a contradiction, $\gamma $ has a symmetric arc between $u_{0}$ and
$u_{1},$ completing the proof of the claim. \hfill $\bigtriangleup $

To finish the proof of the theorem, we consider two cases according
to Claim 1.

\textsc{Case 1.} $j=l-2.$ By Claim 2, $y_{0}\in \mathcal{A}$ and then $%
y_{l-1}\notin \mathcal{A}$. So, there exists an arc between $y_{l-1}$ and $%
u_{1}.$ By the maximality of $j,$ there exists $(y_{l-1},u_{1})\in
A(D).$ Hence the directed cycle
\begin{equation}
\overrightarrow{C}_{3}\cong (u_{1},y_{l-2},y_{l-1},u_{1})
\label{3-cycle-2}
\end{equation}%
is monochromatic by hypothesis and then there exists $u_{1}\leadsto
_{k}y_{0} $ with $k\leq 2.$ If $y_{0}=u_{0},$ then we have a
contradiction, $\gamma $ has a symmetric arc between $u_{0}$ and
$u_{1}.$ Therefore $y_{1}=u_{0}.$ On
the other hand, there exists an arc between $y_{l-2}$ and $y_{0}$ since $%
y_{l-2}\notin \mathcal{A}$ and $y_{0}\in \mathcal{A}$. If $%
(y_{0},y_{l-2})\in A(D),$ then the directed cycle $\overrightarrow{C}%
_{3}\cong (y_{l-1},y_{0},y_{l-2},y_{l-1})$ is monochromatic and of
the same color as the $\overrightarrow{C}_{3}$ of (\ref{3-cycle-2}).
Thus, there
exists $u_{1}\leadsto _{k}u_{0}$ with $k\leq 2,$ particularly,%
\[
u_{1}\longrightarrow y_{l-2}\longrightarrow y_{l-1}\longrightarrow
y_{0}\longrightarrow y_{1}=u_{0},
\]%
we have the same contradiction as previously. In consequence, $%
(y_{l-2},y_{0})\in A(D).$ In addition, since $y_{0}\in \mathcal{A}$ and $%
y_{1}=u_{0}\notin \mathcal{A}$, there exists the arc
$(u_{0},u_{1})\in A(D)$ (otherwise, $(u_{1},u_{0})\in A(D)$ yielding
the same contradiction as
before). Then the directed cycle%
\[
\overrightarrow{C}_{4}\cong
(y_{1}=u_{0},u_{1},y_{l-2},y_{0},y_{1}=u_{0})
\]%
is monochromatic and there exists a monochromatic $u_{1}\leadsto
u_{0},$ the same contradiction once more.

\textsc{Case 2.} $j=l-1.$ In this case, there exists the directed path $%
u_{1}\longrightarrow y_{l-1}\longrightarrow y_{0}.$ So, $y_{0}\neq
u_{0}$,
otherwise we have a contradiction, $\gamma $ has a symmetric arc between $%
u_{0}$ and $u_{1}.$ Hence $y_{1}=u_{0}$ and since $y_{0}\in
\mathcal{A}$ by
Claim 2, $y_{1}=u_{0}\notin \mathcal{A}$. Thus, there exists an arc between $%
u_{0}$ and $u_{1}$ which should be $(u_{0},u_{1})\in A(D)$ (if not, $%
(u_{1},u_{0})\in A(D)$ and we have a contradiction). Therefore the
directed cycle
\[
\overrightarrow{C}_{4}\cong
(y_{1}=u_{0},u_{1},y_{l-1},y_{0},y_{1}=u_{0})
\]%
is monochromatic and there exists a monochromatic $u_{1}\leadsto
u_{0},$ the same contradiction once more.

This concludes the proof of the theorem.
\end{proof}

In a very similar way as the proofs of the above theorems, we can
show the following theorem for semicomplete bipartite digraphs.

\begin{theorem}

Let $D$ be an $m$-colored semicomplete bipartite digraph and $k=2$ (resp. $%
k=3$). If every $\overrightarrow{C}_{4}\upuparrows
\overrightarrow{C}_{4}$
contained in $D$ is at most $2$-colored (resp. $3$-colored), then $D$ has a $%
2$-colored (resp. $3$-colored) kernel. \label{bipart-k-equal-23}
\end{theorem}

We summarize the known results on the existence of $k$-colored
kernels for $m$-colored semicomplete multipartite digraphs and
multipartite tournaments in the next two corollaries.

\begin{corollary}

Let $D$ be an $m$-colored semicomplete $r$-partite digraph and
$r\geq 2.$

\begin{enumerate}

\item[(i)] If $r\geq 3,$ $k=2$ and every $\overrightarrow{C}_{3}$ and $%
\overrightarrow{C}_{4}$ contained in $D$ is monochromatic, then $D$ has a $2$%
-colored kernel (Theorem \ref{k-equal-2}).

\item[(ii)] If $r\geq 3,$ $k=3$ and every $\overrightarrow{C}_{4}$
contained in $D$ is at most $2$-colored and, either every $\overrightarrow{C}%
_{5}$ contained in $D$ is at most $3$-colored or every $\overrightarrow{C}%
_{3}\uparrow \overrightarrow{C}_{3}$ contained in $D$ is at most $2$%
-colored, then $D$ has a $3$-colored kernel (Theorem
\ref{k-equal-3}).

\item[(iii)] If $r\geq 2$ and $k\geq 4,$ then $D$ has a $k$-colored kernel
(Theorem \ref{k-gr-equal-4} and Theorem 14 of \cite{GS-Ll-MB}).

\item[(iv)] If $r=2,$ $k=2$ (resp. $k=3)$ and every $\overrightarrow{C}%
_{4}\upuparrows \overrightarrow{C}_{4}$ contained in $D$ is at most $2$%
-colored, then $D$ has a $2$-colored (resp. $3$-colored) kernel (Theorem \ref%
{bipart-k-equal-23}).

\end{enumerate} \label{final-cor}

\end{corollary}

\begin{corollary}

Let $D$ be an $m$-colored $r$-partite tournament with $r\geq 2.$
Then the conclusions (i)--(iv) of Corollary \ref{final-cor} remain
valid. Moreover,

\begin{enumerate}

\item[(i)] if $r=2,$ $k=1$ and every $\overrightarrow{C}_{4}$ contained in $D
$ is monochromatic, then $D$ has a $1$-colored kernel (Theorem 2.1
of \cite{GS-RM1}), and

\item[(ii)] if $r\geq 3$, $k=1$ and every $\overrightarrow{C}_{3}$ and $%
\overrightarrow{C}_{4}$ contained in $D$ is monochromatic, then $D$ has a $1$%
-colored kernel (Theorem 3.3 of \cite{GS-RM2}).

\end{enumerate} \label{final-cor2}

\end{corollary}

We conclude this paper with the following challenging conjecture. If
it were true, the resulting theorem would be a fine generalization
of Theorem 3.3 proved in \cite{GS-RM2}.

\begin{conjecture}

Let $D$ be an $m$-colored semicomplete $r$-partite digraph with
$r\geq 2.$ If every $\overrightarrow{C}_{3}$ and $%
\overrightarrow{C}_{4}$ contained in $D$ is monochromatic, then $D$ has a $1$%
-colored kernel.

\end{conjecture}

\end{document}